%% file: main.tex
\documentclass[hidelinks,onefignum,onetabnum]{siamart220329}
\usepackage{amsmath}
\usepackage{amssymb}
\usepackage{enumerate}
\usepackage{mathrsfs}
\usepackage{cases}
\usepackage{color}
\usepackage{tikz}
\usepackage{mathtools}
\usepackage{makecell}
\usepackage{multirow}
\usepackage{siunitx}
\usepackage{graphicx}
\definecolor{darkblue}{rgb}{0.0, 0.0, 0.55}
\definecolor{bordeaux}{rgb}{0.34, 0.01, 0.1}

\newsiamthm{example}{Example}

\input{ex_shared}

\def\R{{\mathbb{R}}}
\def\C{{\mathbb{C}}}
\def\N{{\mathbb{N}}}

\def\z{{\mathbf{z}}}

\def\a{{\boldsymbol{\alpha}}}

\def\b{{\boldsymbol{\beta}}}
\def\g{{\boldsymbol{\gamma}}}

\def\A{{\mathscr{A}}}

\def\S{{\mathbf{S}}}

\def\H{{\mathrm{H}}}
\def\M{{\mathbf{M}}}

\def\i{\hbox{\bf{i}}}
\def\supp{\hbox{\rm{supp}}}

\def\Tr{\hbox{\rm{Tr}}}

\def\II{{\mathcal{I}}}
\def\RR{{\mathcal{R}}}
\def\HH{{\mathcal{H}}}

\begin{document}

\maketitle

\begin{abstract}
This note proposes a new reformulation of complex semidefinite programs (SDP) as real SDPs. As an application, we present an economical reformulation of complex SDP relaxations of complex polynomial optimization problems as real SDPs and derive some further reductions by exploiting inner structure of the complex SDP relaxations. Various numerical examples demonstrate that our new reformulation runs significantly faster than the usual popular reformulation. 
\end{abstract}



\section{Introduction}
Complex semidefinite programs (SDP) arise from a diverse set of areas, such as combinatorial optimization \cite{goemans2001approximation}, optimal power flow \cite{josz2018lasserre,oustry2022certified}, quantum information theory \cite{araujo2023quantum,bae2015quantum,Watrous09}, and signal processing \cite{lu2019tightness,waldspurger2015phase}. In particular, they appear as convex relaxations of complex polynomial optimization problems (CPOP), giving rise to the complex moment-Hermitian-sum-of-squares (moment-HSOS) hierarchy \cite{josz2018lasserre,wang2022exploiting,wang2023real}. However, most modern SDP solvers deal with only real SDPs\footnote{As far as the author knows, {\tt SeDuMi} \cite{sedumi}, {\tt Sdolab} \cite{gilbert2017sdolab}, and {\tt Hypatia} \cite{coey2022solving} are the only solvers that can handle complex SDPs directly.}. In order to handle complex SDPs via real SDP solvers, it is mandatory to reformulate complex SDPs as equivalent real SDPs. A popular way\footnote{See for instance the online modeling cookbook of the commercial SDP solver {\tt MOSEK}: \url{https://docs.mosek.com/modeling-cookbook/sdo.html}.} to do so is to use the equivalent condition
\begin{equation}\label{sec1:eq1}
    H\succeq0\quad \iff \quad Y=\begin{bmatrix}H_R&-H_I\\H_I&H_R
    \end{bmatrix}\succeq0
\end{equation}
for a Hermitian matrix variable $H=H_R+H_I\i\in\C^{n\times n}$ with $H_R$ and $H_I$ being its real and imaginary parts respectively.
Note that the right-hand-side constraint in \eqref{sec1:eq1} entails certain structure and to feed it to an SDP solver, we need to impose extra affine constraints to the positive semidefinite (PSD) constraint $Y\succeq0$:
\begin{equation}\label{sec1:eq2}
    Y_{i,j}=Y_{i+n,j+n},\, Y_{i,j+n}+Y_{j,i+n}=0,\quad i=1,\ldots,n,\, j=i,\ldots,n.
\end{equation}
This conversion is quite simple but could be inefficient when $n$ is large. In this note, inspired by Lagrange duality, we propose a new reformulation of complex SDPs as real SDPs. The benefit of this new reformulation is that it contains no extra affine constraints while keeping the same matrix size and hence has lower complexity. In the same manner, we can obtain a new reformulation of complex SDP relaxations of CPOPs as real SDPs. Moreover, by exploiting inner structure of the complex SDP relaxations, we are able to remove a bunch of redundant affine constraints, which leads to an even more economical real reformulation of the complex SDP relaxations.
Various numerical experiments (including the complex matrix completion problem, complex SDP relaxations for randomly generated CPOPs, and the alternating current optimal power flow (AC-OPF) problem) confirm our theoretical expectation and demonstrate that the new reformulation is substantially more efficient than the usual popular one. Finally, numerical experiments also indicate that our implementation of the new reformulation with {\tt MOSEK} \cite{mosek} runs much faster than the implementation of the original complex formulation with {\tt Hypatia} \cite{gilbert2017plea}, probably because {\tt MOSEK} is commercial and the implementations of SDP algorithms based on real numbers are more mature and robust.

\vspace{1em}
\noindent{\bf Notation.} The symbol $\N$ denotes the set of nonnegative integers. For $n\in\N\setminus\{0\}$, let $[n]\coloneqq\{1,2,\ldots,n\}$. We write $|A|$ for the cardinality of a set $A$. Let $\i$ be the imaginary unit, satisfying $\i^2 = -1$. For $d\in\N$, let $\N^n_d\coloneqq\{(\alpha_i)_{i}\in\N^n\mid\sum_{i=1}^n\alpha_i\le d\}$ and let $\omega_{n,d}\coloneqq\binom{n+d}{d}$ be the cardinality of $\N^n_d$. For $\a=(\alpha_i)_i\in\N^n_d$ and an $n$-tuple of variables $\z=\{z_1,\ldots,z_n\}$, let $\z^{\a}\coloneqq z_1^{\alpha_1}\cdots z_n^{\alpha_n}$. For a complex number $a$, $\overline{a}$ (resp.\ $\RR(a),\II(a)$) denotes the conjugate (resp.\ real part, imaginary part) of $a$, and for a complex vector $v$, $v^{\H}$ denotes the conjugate transpose of $v$. For a positive integer $n$, the set of $n\times n$ symmetric (resp.\ Hermitian) matrices is denoted by $\S^n$ (resp.\ $\H^n$). We use $A\succeq0$ to indicate that the matrix $A$ is PSD. For $A,B\in\C^{n\times n}$, we denote by $\langle A,B\rangle$ the trace inner-product, defined by $\langle A, B\rangle=\Tr(A^{\H}B)$, where $A^{\H}$ stands for the conjugate transpose of $A$. For $A\in\R^{n\times n}$, $A^{\intercal}$ stands for the transpose of $A$.

We endow $\C^m$ (viewed as a $\R$-vector space) with the scalar product $\langle \cdot, \cdot\rangle_{\R}$ defined by
\begin{equation*}
    \langle u, v\rangle_{\R}=\RR(u^{\H}v)=\RR(u)^{\intercal}\RR(v)+\II(u)^{\intercal}\II(v), \quad u,v\in\C^m,
\end{equation*}
which will be used in forming the dual complex SDP problem (to obtain a real objective).
For $u,v\in\R^m$, $\langle u, v\rangle\coloneqq u^{\intercal}v$, where $u^{\intercal}$ stands for the transpose of $u$.
We associate each $A\in\C^{n\times n}$ with a Hermitian matrix $\HH(A)\coloneqq\frac{1}{2}(A+A^{\H})$. One can check that $\RR(\langle A, H\rangle)=\langle \HH(A), H\rangle$ for any $H\in\H^n$.

\section{The real reformulations of complex SDPs}\label{sec:main}
Given a tuple of complex matrices $A_1,\ldots,A_m\in\C^{n\times n}$, we define a $\R$-linear operator $\A\colon\H^n\rightarrow\C^m$ by 
\begin{equation}
\A(H)\coloneqq(\langle A_i, H\rangle)_{i=1}^m\in\C^m, \quad\forall H\in\H^n.    
\end{equation}

Let us consider the following complex SDP:
\begin{equation}\label{csdpp}
    \begin{cases}
     \sup\limits_{H\in\H^n} &\langle C,H\rangle\\
      \,\,\,\,\rm{s.t.}&\A(H)= b,\\
      &H\succeq0,
    \end{cases}\tag{PSDP-\mbox{$\C$}}
\end{equation}
where $C\in\H^n,b\in\C^m$. In order to convert \eqref{csdpp} to a real SDP, we define two real linear operators $\A_R,\A_I\colon\R^{n\times n}\rightarrow\R^m$ associated with $\A$ by
\begin{equation}\label{realA}
    \A_R(S)\coloneqq\left(\langle \RR(A_i), S\rangle\right)_{i=1}^m\in\R^m, \quad\forall S\in\R^{n\times n}
\end{equation}
and
\begin{equation}\label{comA}
    \A_I(S)\coloneqq\left(\langle \II(A_i), S\rangle\right)_{i=1}^m\in\R^m, \quad\forall S\in\R^{n\times n},
\end{equation}
respectively. Moreover, assume $H=H_R+H_I\i,C=C_R+C_I\i,b=b_R+b_I\i$ with $H_R,H_I,C_R,C_I\in\R^{n\times n}$, $b_R,b_I\in\R^m$. We can now convert \eqref{csdpp} to a real SDP:
\begin{equation}\label{csdpp1}
    \begin{cases}
     \sup\limits_{Y\in\S^{2n}} &\langle C_R,H_R\rangle+\langle C_I,H_I\rangle\\
      \,\,\,\,\rm{s.t.}&\A_R(H_R)+\A_I(H_I)= b_R,\\
      &\A_R(H_I)-\A_I(H_R)= b_I,\\
      &Y=\begin{bmatrix}H_R&-H_I\\
      H_I&H_R\end{bmatrix}\succeq0.
    \end{cases}\tag{PSDP-\mbox{$\R$}}
\end{equation}

As mentioned in the introduction, to feed the PSD constraint in \eqref{csdpp1} to an SDP solver, we need to include also the extra $n(n+1)$ affine constraints listed in \eqref{sec1:eq2}, which could be inefficient in practice. Below we show that by taking a dual point of view, we can actually get rid of this issue.

Before formulating the dual problem of \eqref{csdpp}, we explicitly give the adjoint operator of $\A$.
\begin{lemma}\label{lm1}
The adjoint operator $\A^*$ of $\A$ satisfies $\A^*(y)=\HH\left(\sum_{i=1}^my_iA_i\right)$ for $y\in\C^m$.
\end{lemma}
\begin{proof}
For any $H\in\H^n$, we have
\begin{align*}
\left\langle \A^*(y),H\right\rangle&=\left\langle y,\A(H)\right\rangle_{\R}=\RR\left(\sum_{i=1}^m\overline{y}_i\langle A_i,H\rangle\right)\\
&=\RR\left(\left\langle\sum_{i=1}^m y_iA_i,H\right\rangle\right)=\left\langle \HH\left(\sum_{i=1}^my_iA_i\right),H\right\rangle,
\end{align*}
which yields $\A^*(y)=\HH\left(\sum_{i=1}^my_iA_i\right)$.
\end{proof}

Then by the convex duality theory, the dual problem of \eqref{csdpp} reads as
\begin{equation}\label{csdpd}
    \begin{cases}
     \inf\limits_{y\in\C^m} &\langle b, y\rangle_{\R}\\
      \,\,\,\rm{s.t.}&\A^*(y)\succeq C.
    \end{cases}\tag{DSDP-\mbox{$\C$}}
\end{equation}
Assume $y=y_R+y_I\i$ with $y_R,y_I\in\R^m$. Using Lemma \ref{lm1}, we deduce that $\A^*(y)=U+V\i$ with $$U\coloneqq\frac{1}{2}\sum_{i=1}^m\RR(y_i)\RR(A_i+A_i^{\intercal})-\II(y_i)\II(A_i+A_i^{\intercal})$$ and
$$V\coloneqq\frac{1}{2}\sum_{i=1}^m\II(y_i)\RR(A_i-A_i^{\intercal})+\RR(y_i)\II(A_i-A_i^{\intercal}).$$
Thus, we can convert \eqref{csdpd} to a real SDP by using the equivalent condition \eqref{sec1:eq1}:
\begin{equation}\label{csdpd1}
    \begin{cases}
     \inf\limits_{y_R,y_I\in\R^m} &b_R^{\intercal}y_R+b_I^{\intercal}y_I\\
      \,\,\,\,\,\,\,\,\rm{s.t.}&\begin{bmatrix}
      U-C_R&-V+C_I\\
      V-C_I&U-C_R
      \end{bmatrix}\succeq0.
    \end{cases}\tag{DSDP-\mbox{$\R$}}
\end{equation}
Let $X=\left[\begin{smallmatrix}X_1&X_3^{\intercal}\\X_3&X_2\end{smallmatrix}\right]\in\S^{2n}$ be the PSD variable of \eqref{csdpd1} with $X_1,X_2,X_3\in\R^{n\times n}$. 
Then the Lagrangian associated with \eqref{csdpd1} is given by
\begin{align*}
    &\,L(X,y_R,y_I)\\
    =&\,b_R^{\intercal}y_R+b_I^{\intercal}y_I-\left\langle \begin{bmatrix}
      X_1&X_3^{\intercal}\\
      X_3&X_2
      \end{bmatrix}, \begin{bmatrix}
      U-C_R&-V+C_I\\
      V-C_I&U-C_R
      \end{bmatrix}\right\rangle\\
      =&\,b_R^{\intercal}y_R+b_I^{\intercal}y_I-\langle X_1+X_2, U-C_R\rangle - \langle X_3-X_3^{\intercal},V-C_I\rangle\\
      =&\,b_R^{\intercal}y_R+b_I^{\intercal}y_I+\langle C_R,X_1+X_2 \rangle+\langle C_I,X_3-X_3^{\intercal}\rangle\\
      &\,-\sum_{i=1}^m\langle \RR(A_i), X_1+X_2\rangle\RR(y_i)
      +\sum_{i=1}^m\langle\II(A_i),X_1+X_2\rangle\II(y_i)\\
      &\,-\sum_{i=1}^m\langle\RR(A_i),X_3-X_3^{\intercal}\rangle\II(y_i)-\sum_{i=1}^m\langle\II(A_i), X_3-X_3^{\intercal}\rangle\RR(y_i)\\
      =&\,\langle C_R,X_1+X_2 \rangle+\langle C_I,X_3-X_3^{\intercal}\rangle+\langle b_R-\A_R(X_1+X_2)-\A_I(X_3-X_3^{\intercal}), y_R\rangle\\
      &\,+\langle b_I-\A_R(X_3-X_3^{\intercal})+\A_I(X_1+X_2), y_I\rangle.
\end{align*}
Therefore, the dual problem of \eqref{csdpd1} can be written down as
\begin{equation}\label{csdpp2}
    \begin{cases}
     \sup\limits_{X\in\S^{2n}} &\langle C_R,X_1+X_2 \rangle+\langle C_I,X_3-X_3^{\intercal} \rangle\\
      \,\,\,\,\rm{s.t.}&\A_R(X_1+X_2)+\A_I(X_3-X_3^{\intercal})=b_R,\\
      &\A_R(X_3-X_3^{\intercal})-\A_I(X_1+X_2)=b_I,\\
      &X=\begin{bmatrix}
      X_1&X_3^{\intercal}\\
      X_3&X_2
      \end{bmatrix}\succeq0.
    \end{cases}\tag{PSDP-\mbox{$\R$}'}
\end{equation}

\begin{theorem}\label{thm1}
\eqref{csdpp2} is equivalent to \eqref{csdpp1} (in the sense that they share the same optimum). As a result, \eqref{csdpp2} is equivalent to \eqref{csdpp}. In addition, if $X^{\star}=\left[\begin{smallmatrix}X_1^{\star}&(X_3^{\star})^{\intercal}\\X_3^{\star}&X_2^{\star}\end{smallmatrix}\right]$ is an optimal solution to \eqref{csdpp2}, then $H^{\star}=(X^{\star}_1+X_2^{\star})+(X^{\star}_3-(X_3^{\star})^{\intercal})\i$ is an optimal solution to \eqref{csdpp}.
\end{theorem}
\begin{proof}
Let us denote the optima of \eqref{csdpp1} and \eqref{csdpp2} by $v$ and $v'$, respectively.
Suppose that $Y=\left[\begin{smallmatrix}H_R&-H_I\\H_I&H_R\end{smallmatrix}\right]$ is a feasible solution to \eqref{csdpp1}. Then one can easily check that $X\coloneqq\frac{1}{2}Y$ is a feasible solution to \eqref{csdpp2}. Moreover, we have $\langle C_R,X_1+X_2 \rangle+\langle C_I,X_3-X_3^{\intercal} \rangle=\langle C_R,H_R\rangle+\langle C_I,H_I\rangle$ and it follows $v\le v'$. On the other hand, suppose $X=\left[\begin{smallmatrix}X_1&X_3^{\intercal}\\X_3&X_2\end{smallmatrix}\right]$ is a feasible solution to \eqref{csdpp2}. Since PSDness is preserved under basis transformation, we have
\begin{equation*}
\begin{bmatrix}0&-1\\1&0\end{bmatrix}^{-1}X\begin{bmatrix}0&-1\\1&0\end{bmatrix}=\begin{bmatrix}X_2&-X_3\\-X_3^{\intercal}&X_1\end{bmatrix}\succeq0,
\end{equation*}
and since PSDness is preserved under addition, we further have
\begin{equation*}
Y=\begin{bmatrix}H_R&-H_I\\H_I&H_R\end{bmatrix}\coloneqq\begin{bmatrix}X_1+X_2&X_3^{\intercal}-X_3\\X_3-X_3^{\intercal}&X_1+X_2\end{bmatrix}\succeq0.
\end{equation*}
One can easily see that $Y$ is a feasible solution to \eqref{csdpp1} and in addition, it holds $\langle C_R,H_R\rangle+\langle C_I,H_I\rangle=\langle C_R,X_1+X_2 \rangle+\langle C_I,X_3-X_3^{\intercal} \rangle$. Thus $v\ge v'$, which proves the equivalence. The latter statement of the theorem is clear from the above arguments.
\end{proof}

In contrast to \eqref{csdpp1}, the PSD constraint in \eqref{csdpp2} does not entails any special structure, and thus no extra affine constraint is required to describe it. This is why the conversion \eqref{csdpp2} is more appealing than \eqref{csdpp1} from the computational perspective.

\begin{remark}
A similar reformulation to \eqref{csdpp2} but for a restricted class of complex SDP relaxations of multiple-input multiple-output detection has appeared in \cite{lu2019tightness}.
\end{remark}

\section{Application to complex SDP relaxations for CPOPs}
In this section, we apply the reformulation \eqref{csdpp2} to complex SDP relaxations arising from the complex moment-HSOS hierarchy for CPOPs. A CPOP is given by
\begin{equation}\label{cpop}
\begin{cases}
\inf\limits_{\z\in\C^s} &f(\z,\overline{\z})=\sum_{\b,\g}b_{\b,\g}\z^{\b}\overline{\z}^{\g}\\
\,\,\rm{s.t.}&g_i(\z,\overline{\z})=\sum_{\b,\g}g^i_{\b,\g}\z^{\b}\overline{\z}^{\g}\ge0,\quad i\in[t],
\end{cases}\tag{CPOP}
\end{equation}
where $\overline{\z}\coloneqq(\overline{z}_1,\ldots,\overline{z}_s)$ stands for the conjugate of complex variables $\z\coloneqq(z_1,\ldots,z_s)$. The functions $f,g_1,\ldots,g_t$ are real-valued polynomials in $\z$ and $\overline{\z}$ (their coefficients satisfy $b_{\b,\g}=\overline{b_{\g,\b}}$, $g^i_{\b,\g}= \overline{g^i_{\g,\b}}$). The \emph{support} of $f$ is defined by $\supp(f)\coloneqq\{(\b,\g)\mid b_{\b,\g}\ne0\}$. For $i\in[t]$, $\supp(g_i)$ is defined in the same way.

Fix a $d\in\N$. Let $y=(y_{\b,\g})_{(\b,\g)\in\N_d^s\times\N_d^s}$ be a complex sequence indexed by $(\b,\g)\in\N_d^s\times\N_d^s$ and satisfying $y_{\b,\g}=\overline{y_{\g,\b}}$. Let $L_{y}$ be the linear functional defined by
\begin{equation*}
f=\sum_{(\b,\g)}b_{\b,\g}\z^{\b}\overline{\z}^{\g}\mapsto L_{y}(f)=\sum_{(\b,\g)}b_{\b,\g}y_{\b,\g}.
\end{equation*}
The \emph{complex moment} matrix $\M_{d}(y)$ associated with $y$ is the Hermitian matrix indexed by $\N^s_d$ such that
\begin{equation*}
[\M_d(y)]_{\b,\g}\coloneqq L_{y}(\z^{\b}\overline{\z}^{\g})=y_{\b,\g}, \quad\forall\b,\g\in\N^s_d.
\end{equation*}
Suppose that $g=\sum_{(\b',\g')}g_{\b',\g'}\z^{\b'}\overline{\z}^{\g'}$ is a complex polynomial. The \emph{complex localizing} matrix $\M_{d}(gy)$ associated with $g$ and $y$ is the Hermitian matrix indexed by $\N^s_d$ such that
\begin{equation*}
[\M_{d}(gy)]_{\b,\g}\coloneqq L_{y}(g\z^{\b}\overline{\z}^{\g})=\sum_{(\b',\g')}g_{\b',\g'}y_{\b+\b',\g+\g'}, \quad\forall\b,\g\in\N^s_d.
\end{equation*}

Let $d_0\coloneqq\max\,\{|\b|,|\g|: b_{\b,\g}\ne0\}$ and $d_i\coloneqq\max\,\{|\b|,|\g|: g^i_{\b,\g}\ne0\}$ for $i\in[t]$, where $|\cdot|$ denotes the sum of entries. Let $d_{\min}\coloneqq\max\,\{d_0,d_1,\ldots,d_m\}$.
For any $d\ge d_{\min}$, the $d$-th ($d$ is called the \emph{relaxation order}) complex moment relaxation for \eqref{cpop} is given by
\begin{equation}\label{sec2-eq1}
\begin{cases}
\inf\limits_{y}& L_y(f)=\langle b, y\rangle_{\R}\\
\rm{s.t.}&\M_{d}(y)\succeq0,\\
&\M_{d-d_i}(g_iy)\succeq0,\quad i\in[t],\\
&y_{\mathbf{0},\mathbf{0}}=1.
\end{cases}\tag{Mom-\mbox{$\C$}}
\end{equation}
\eqref{sec2-eq1} and its dual problem form the complex moment-HSOS hierarchy of \eqref{cpop}. For more details on this hierarchy, we refer the reader to \cite{josz2018lasserre,wang2022exploiting}.

For any $(\b,\g)\in\N^s_d\times\N^s_d$, we associate it with a matrix $A^0_{\b,\g}\in\R^{\omega_{s,d}\times\omega_{s,d}}$ defined by
\begin{equation}
    [A^0_{\b,\g}]_{\b',\g'}=\begin{cases}1,&\text{if } (\b',\g')=(\b,\g),\\
    0,&\text{otherwise}.
    \end{cases}
\end{equation}
Moreover, for each $i\in[t]$, we associate any $(\b,\g)\in\N^s_{d-d_i}\times\N^s_{d-d_i}$ with a matrix $A^i_{\b,\g}\in\C^{\omega_{s,d-d_i}\times\omega_{s,d-d_i}}$ defined by
\begin{equation}
[A^i_{\b,\g}]_{\b',\g'}=
\begin{cases}
g^i_{\b'',\g''},&\text{if } (\b'+\b'',\g'+\g'')=(\b,\g),\\
0,&\text{otherwise}.
\end{cases}
\end{equation}
Now for each $i=0,1,\ldots,t$, we define the $\R$-linear operator $\A^i$ by 
\begin{equation*}
    \A^i(H)\coloneqq\left(\langle A^i_{\b,\g}, H\rangle\right)_{(\b,\g)\in\N^s_{d-d_i}\times\N^s_{d-d_i}},\quad H\in\H^{\omega_{s,d-d_i}}.
\end{equation*}
For convenience let us set $g_0\coloneqq1$. By construction, it holds
\begin{equation*}
\M_{d-d_i}(g_iy)=\sum_{(\b,\g)\in\N^s_{d-d_i}\times\N^s_{d-d_i}}A_{\b,\g}^iy_{\b,\g}=(\A^i)^*(y),\quad i=0,1,\ldots,t.
\end{equation*}
Therefore, one can rewrite \eqref{sec2-eq1} as follows: 
\begin{equation}\label{cmom}
    \begin{cases}
     \inf\limits_{y}&\langle b, y\rangle_{\R}\\
      \rm{s.t.}
      &(\A^i)^*(y)\succeq0,\quad i=0,1,\ldots,t,\\
      &y_{\mathbf{0},\mathbf{0}}=1,
    \end{cases}\tag{Mom-\mbox{$\C$}'}
\end{equation}
whose dual reads as
\begin{equation}\label{hsos}
    \begin{cases}
     \sup\limits_{\lambda,H^i} &\lambda\\
      \,\rm{s.t.}&\sum_{i=0}^t[\A^i(H^i)]_{\b,\g}+\delta_{(\b,\g),\,(\mathbf{0},\mathbf{0})}\lambda=b_{\b,\g},\quad(\b,\g)\in\N^s_d\times\N^s_d,\\
      &H^i\succeq0,\quad i=0,1,\ldots,t.
    \end{cases}\tag{HSOS-\mbox{$\C$}}
\end{equation}
Note that we have used the Kronecker delta $\delta_{(\b,\g),\,(\mathbf{0},\mathbf{0})}$ in \eqref{hsos}.

Let us from now on fix any order ``$<$" on $\N^s$.
\begin{proposition}
\eqref{hsos} is equivalent to the following complex SDP:
\begin{equation}\label{hsos1}
    \begin{cases}
     \sup\limits_{\lambda,H^i} &\lambda\\
      \,\rm{s.t.}&\sum_{i=0}^t[\A^i(H^i)]_{\b,\g}+\delta_{(\b,\g),\,(\mathbf{0},\mathbf{0})}\lambda=b_{\b,\g},\\
      &\b\le\g,\quad (\b,\g)\in\N^s_d\times\N^s_d,\\
      &H^i\succeq0,\quad i=0,1,\ldots,t.
    \end{cases}\tag{HSOS-\mbox{$\C$}'}
\end{equation}
\end{proposition}
\begin{proof}
It suffices to show that for $\b<\g$, $\sum_{i=0}^t[\A^i(H^i)]_{\g,\b}=b_{\g,\b}$ is equivalent to $\sum_{i=0}^t[\A^i(H^i)]_{\b,\g}=b_{\b,\g}$. Indeed, this equivalence follows from $b_{\g,\b}=\overline{b_{\b,\g}}$ and 
\begin{align*}
    \sum_{i=0}^t[\A^i(H^i)]_{\g,\b}&=\sum_{i=0}^t\langle A^i_{\g,\b}, H^i\rangle\\
    &=[H^0]_{\g,\b}+\sum_{i=1}^t\sum_{\substack{(\g',\b')\in\N^s_{d-d_i}\times\N^s_{d-d_i}\\(\g'',\b'')\in\supp(g)\\(\g'+\g'',\b'+\b'')=(\g,\b)}}g^i_{\g'',\b''}[H^i]_{\g',\b'}\\
    &=\sum_{i=0}^t\langle \overline{A^i_{\b,\g}}, H^i\rangle=\sum_{i=0}^t\overline{[\A^i(H^i)]_{\b,\g}}.
\end{align*}
\end{proof}

Let $\A^i_R$ and $\A^i_I$ be the real linear operators associated with $\A^i$ which are defined in a similar way as \eqref{realA} and \eqref{comA}. With $H^i=H^i_R+H^i_I\i,b=b_R+b_I\i$, \eqref{hsos1} is equivalent to the following real SDP by using the equivalent condition \eqref{sec1:eq1}:
\begin{equation}\label{rsos0}
    \begin{cases}
     \sup\limits_{\lambda,Y^i} &\lambda\\
     \,\rm{s.t.}&\sum_{i=0}^t\left([\A^i_R(H^i_R)]_{\b,\g}+[\A^i_I(H^i_I)]_{\b,\g}\right)+\delta_{(\b,\g),\,(\mathbf{0},\mathbf{0})}\lambda=[b_R]_{\b,\g},\\
      &\sum_{i=0}^t\left([\A^i_R(H^i_I)]_{\b,\g}-[\A^i_I(H^i_R)]_{\b,\g}\right)=[b_I]_{\b,\g},\\
      &\b\le\g,\quad(\b,\g)\in\N^s_d\times\N^s_d,\\
      &Y^i=\begin{bmatrix}
      H^i_R&-H^i_I\\
      H^i_I&H^i_R
      \end{bmatrix}\succeq0,\quad i=0,1,\ldots,t.
    \end{cases}\tag{HSOS-\mbox{$\R$}}
\end{equation}
On the other hand, by invoking Theorem \ref{thm1}, we obtain another equivalent real SDP conversion of \eqref{hsos1}:
\begin{equation}\label{rsos}
    \begin{cases}
     \sup\limits_{\lambda,X^i} &\lambda\\
     \,\rm{s.t.}&\sum_{i=0}^t\left([\A^i_R(X^i_1+X^i_2)]_{\b,\g}+[\A^i_I(X^i_3-(X^i_3)^{\intercal})]_{\b,\g}\right)+\delta_{(\b,\g),\,(\mathbf{0},\mathbf{0})}\lambda=[b_R]_{\b,\g},\\
      &\sum_{i=0}^t\left([\A^i_R(X^i_3-(X^i_3)^{\intercal})]_{\b,\g}-[\A^i_I(X^i_1+X^i_2)]_{\b,\g}\right)=[b_I]_{\b,\g},\\
      &\b\le\g,\quad(\b,\g)\in\N^s_d\times\N^s_d,\\
      &X^i=\begin{bmatrix}
      X^i_1&(X^i_3)^{\intercal}\\
      X^i_3&X^i_2
      \end{bmatrix}\succeq0,\quad i=0,1,\ldots,t.
    \end{cases}
\end{equation}

\begin{proposition}
\eqref{rsos} is equivalent to the following real SDP:
\begin{equation}\label{rsos1}
    \begin{cases}
     \sup\limits_{\lambda,X^i} &\lambda\\
     \,\rm{s.t.}&\sum_{i=0}^t\left([\A^i_R(X^i_1+X^i_2)]_{\b,\g}+[\A^i_I(X^i_3-(X^i_3)^{\intercal})]_{\b,\g}\right)+\delta_{(\b,\g),\,(\mathbf{0},\mathbf{0})}\lambda=[b_R]_{\b,\g},\\
      &\sum_{i=0}^t\left([\A^i_R(X^i_3-(X^i_3)^{\intercal})]_{\b,\g}-[\A^i_I(X^i_1+X^i_2)]_{\b,\g}\right)=[b_I]_{\b,\g},\quad\b\ne\g,\\
      &\b\le\g,\quad(\b,\g)\in\N^s_d\times\N^s_d,\\
      &X^i=\begin{bmatrix}
      X^i_1&(X^i_3)^{\intercal}\\
      X^i_3&X^i_2
      \end{bmatrix}\succeq0,\quad i=0,1,\ldots,t.
    \end{cases}\tag{HSOS-\mbox{$\R$}'}
\end{equation}
\end{proposition}
\begin{proof}
We need to show that the following constraints
\begin{equation}\label{redun}
    \sum_{i=0}^t\left([\A^i_R(X^i_3-(X^i_3)^{\intercal})]_{\b,\b}-[\A^i_I(X^i_1+X^i_2)]_{\b,\b}\right)=[b_I]_{\b,\b}=0,\quad\b\in\N^s_d
\end{equation}
in \eqref{rsos} are redundant. For each $i=0,1,\ldots,t$ and $\b\in\N^s_d$, we have
\begin{align*}
&\,\left\langle(A^i_{\b,\b})_R,X^i_3-(X^i_3)^{\intercal}\right\rangle\\
=&\,\sum_{\substack{(\b',\g')\in\N^s_{d-d_i}\times\N^s_{d-d_i}\\(\b'',\g'')\in\supp(g)\\(\b'+\b'',\g'+\g'')=(\b,\b)}}\RR(g^i_{\b'',\g''})\left([X^i_3]_{\b',\g'}-[(X^i_3)^{\intercal}]_{\b',\g'}\right)\\
=&\,\frac{1}{2}\sum_{\substack{(\b',\g')\in\N^s_{d-d_i}\times\N^s_{d-d_i}\\(\b'',\g'')\in\supp(g)\\(\b'+\b'',\g'+\g'')=(\b,\b)}}\RR(g^i_{\b'',\g''})\left([X^i_3]_{\b',\g'}+[X^i_3]_{\g',\b'}-[(X^i_3)^{\intercal}]_{\b',\g'}-[(X^i_3)^{\intercal}]_{\g',\b'}\right)\\
=&\,0,
\end{align*}
where we have used fact that $[(X^i_3)^{\intercal}]_{\b',\g'}=[X^i_3]_{\g',\b'}$ and $[(X^i_3)^{\intercal}]_{\g',\b'}=[X^i_3]_{\b',\g'}$. It follows that $[\A^i_R(X^i_3-(X^i_3)^{\intercal})]_{\b,\b}=\langle(A^i_{\b,\b})_R,X^i_3-(X^i_3)^{\intercal}\rangle=0$.

In addition, for each $i=0,1,\ldots,t$ and $\b\in\N^s_d$, we have
\begin{align*}
&\,\left\langle(A^i_{\b,\b})_I,X^i_1+X^i_2\right\rangle\\
=&\,\sum_{\substack{(\b',\g')\in\N^s_{d-d_i}\times\N^s_{d-d_i}\\(\b'',\g'')\in\supp(g)\\(\b'+\b'',\g'+\g'')=(\b,\b)}}\II(g^i_{\b'',\g''})\left([X^i_1]_{\b',\g'}+[X^i_2]_{\b',\g'}\right)\\
=&\,\frac{1}{2}\sum_{\substack{(\b',\g')\in\N^s_{d-d_i}\times\N^s_{d-d_i}\\(\b'',\g'')\in\supp(g)\\(\b'+\b'',\g'+\g'')=(\b,\b)}}\II\left(g^i_{\b'',\g''}+g^i_{\g'',\b''}\right)\left([X^i_1]_{\b',\g'}+[X^i_2]_{\b',\g'}\right)\\
=&\,0,
\end{align*}
where we have used fact that $\II(g^i_{\b'',\g''}+g^i_{\g'',\b''})=\II(g^i_{\b'',\g''}+\overline{g}^i_{\b'',\g''})=0$ and that $X^i_1,X^i_2$ are symmetric. It follows that $[\A^i_I(X^i_1+X^i_2)]_{\b,\b}=\langle (A^i_{\b,\b})_I, X^i_1+X^i_2\rangle=0$.

Putting all above together yields \eqref{redun}.
\end{proof}

We have proved the following theorem.
\begin{theorem}
\eqref{rsos1} is equivalent to \eqref{hsos}.
\end{theorem}

Before closing the section, we compare complexity of different real SDP reformulations for complex SDP relaxations of \eqref{cpop} in Table \ref{tab0}.
\begin{table}[htbp]
\caption{Comparison of complexity of different real SDP reformulations for complex SDP relaxations of \eqref{cpop}. $n_{\rm{sdp}}$: the maximal size of SDP matrix, $m_{\rm{sdp}}$: the number of affine constraints.}\label{tab0}
\renewcommand\arraystretch{1.2}
\centering
\begin{tabular}{c|c|c}
&\eqref{rsos0}&\eqref{rsos1}\\
\hline
$n_{\rm{sdp}}$&$2\omega_{s,d}$&$2\omega_{s,d}$\\
\hline
$m_{\rm{sdp}}$&$2\omega_{s,d}^2+2\omega_{s,d}+\sum_{i=1}^t\omega_{s,d-d_i}$&$\omega_{s,d}^2$\\
\end{tabular}
\end{table}

\section{Numerical experiments}
In this section, we benchmark the performance of the two real reformulations for complex SDPs using the software {\tt TSSOS} 1.2.1\footnote{{\tt TSSOS} is freely available at \href{https://github.com/wangjie212/TSSOS}{https://github.com/wangjie212/TSSOS}.} in which {\tt MOSEK} 10.0 \cite{mosek} is employed as an SDP solver with default settings. For comparison, we also include the results of directly solving complex SDPs obtained with {\tt Hypatia} 0.8.1 \cite{coey2022solving} in Sections \ref{subsec0}--\ref{subsec3}. We emphasis that the purpose of numerical experiments is to compare different formulations of complex SDPs and a comparison between solvers is not our focus in this paper.
All numerical experiments were performed on a desktop computer with an Intel(R) Core(TM) i9-10900 CPU@2.80GHz and 64G RAM. When presenting the results, the column labelled by `opt' records the optimum and the column labelled by `time' records computational time in seconds. Moreover, the symbol `-' means the SDP solver runs out of memory, and the symbol `$*$' means computational time exceeds 10000 seconds.

\subsection{The complex matrix completion problem}\label{subsec0}
The complex matrix completion problem seeks to recover a low-rank matrix $M\in\C^{s\times t}$ from a subset of entries $\{M_{ij}\}_{(i,j)\in\Omega}$ \cite{candes2012exact,candes2010power}, which could be tackled via the convex optimization problem:
\begin{equation}\label{mc}
\begin{cases}
\inf\limits_{Z\in\C^{s\times t}}&\|Z\|_*\\
\,\,\,\,\,\,\rm{s.t.}&Z_{ij}=M_{ij},\quad\forall(i,j)\in\Omega,
\end{cases}
\end{equation}
where $\|Z\|_*\coloneqq\Tr(Z^{\H}Z)^{\frac{1}{2}}$ is the nuclear norm of $Z$. One can equivalently cast \eqref{mc} as a complex SDP:
\begin{equation}\label{mc1}
\begin{cases}
\inf\limits_{X\in\H_n}&\Tr(X)\\
\,\,\,\rm{s.t.}&Z_{ij}=M_{ij},\quad\forall(i,j)\in\Omega,\\
&X=\begin{bmatrix}U&Z^{\H}\\ Z&V\end{bmatrix}\succeq0.
\end{cases}
\end{equation}
In this subsection, we generate random instances of the complex matrix completion problem \eqref{mc} as follows: 1) select $\Omega\subseteq[s]\times[s]$ uniformly at random from all subsets with cardinality $10s$ for $s\in\{20,30,40,50,60,70\}$; 2) set $M=M_1M_2^{\H}$ with the real and imaginary parts of entries of $M_1\in\C^{s\times 5}$ and $M_2\in\C^{s\times 5}$ being selected i.i.d. from the standard normal distribution.
The related results are shown in Table \ref{tab00}. From the table, we see that the reformulation \eqref{rsos1} is $1\sim3$ magnitudes faster than the reformulation \eqref{rsos0}.

\begin{table}[htbp]
\caption{The complex matrix completion problem.}\label{tab00}
\renewcommand\arraystretch{1.2}
\centering
\begin{tabular}{c|c|c|c|c|c|c|c|c}
\multirow{2}*{$s$}&\multicolumn{3}{c|}{\eqref{csdpp1}}&\multicolumn{3}{c|}{\eqref{csdpp2}}&\multicolumn{2}{c}{\eqref{csdpp}}\\
\cline{2-9}
&$m_{\rm{sdp}}$&opt&time&$m_{\rm{sdp}}$&opt&time&opt&time\\
\hline
20&1948&315.882&0.62&308&315.882&0.06&315.882&3.46\\
\hline
30&4168&456.591&5.38&508&456.591&0.11&456.591&31.1\\
\hline
40&7160&670.611&62.7&680&670.611&0.37&670.611&153\\
\hline
50&11010&861.510&175&910&861.510&0.77&861.510&626\\
\hline
60&15602&925.590&573&1082&925.590&1.07&925.590&1899\\
\hline
70&21056&1108.93&1337&1316&1108.93&1.81&1108.93&4870\\
\end{tabular}
\end{table}

\subsection{Minimizing a random complex quartic polynomial over the unit sphere}\label{subsec1}
Our second example is to minimize a complex quartic polynomial over the unit sphere:
\begin{equation}\label{random:eq1}
\begin{cases}
\inf\limits_{\z\in\C^{s}} &[\z]_2^{\H}Q[\z]_2\\
\,\,\,\rm{s.t.}&|z_1|^2+\cdots+|z_s|^2=1,
\end{cases}
\end{equation}
where $[\z]_2$ is the column vector of monomials in $\z$ up to degree two and $Q\in\H^{|[\z]_2|}$ is a random Hermitian matrix whose entries are selected with respect to the standard normal distribution.

We approach \eqref{random:eq1} for $s=5,7,\ldots,15$ with the second and third HSOS relaxations. The related results are shown in Table \ref{tab1}. From the table, we see that the reformulation \eqref{rsos1} is several ($2\sim7$) times as fast as the reformulation \eqref{rsos0}, and the speedup becomes more significant as the SDP size grows.

\begin{table}[htbp]
\caption{Minimizing a random complex quartic polynomial over the unit sphere.}\label{tab1}
\renewcommand\arraystretch{1.2}
\centering
\resizebox{\linewidth}{!}{
\begin{tabular}{c|c|c|c|c|c|c|c|c|c}
\multirow{2}*{$s$}&\multirow{2}*{$d$}&\multicolumn{3}{c|}{\eqref{rsos0}}&\multicolumn{3}{c|}{\eqref{rsos1}}&\multicolumn{2}{c}{\eqref{hsos1}}\\
\cline{3-10}
&&$m_{\rm{sdp}}$&opt&time&$m_{\rm{sdp}}$&opt&time&opt&time\\
\hline
\multirow{2}{*}{5}&2&966&-11.2409&0.11&441&-11.2409&0.05&-11.2409&0.12\\
\cline{2-10}
&3&6846&-9.47725&8.13&3136&-9.47725&2.00&-9.47725&6.54\\
\hline
\multirow{2}{*}{7}&2&2736&-14.2314&0.97&1296&-14.2314&0.28&-14.2314&0.59\\
\cline{2-10}
&3&30372&-11.0407&389&14400&-11.0407&57.0&-11.0407&474\\
\hline
\multirow{2}{*}{9}&2&6270&-19.0019&5.73&3025&-19.0019&1.62&-19.0019&4.61\\
\cline{2-10}
&3&100320&-&-&48400&-15.5614&1944&-&-\\
\hline
\multirow{2}{*}{11}&2&12480&-22.8630&31.7&6084&-22.8630&6.67&-22.8630&32.3\\
\cline{2-10}
&3&271882&-&-&132496&-&-&-&-\\
\hline
\multirow{2}{*}{13}&2&22470&-25.6352&145&11025&-25.6352&23.5&-25.6352&174\\
\cline{2-10}
&3&639450&-&-&313600&-&-&-&-\\
\hline
\multirow{2}{*}{15}&2&37536&-29.1672&585&18496&-29.1672&86.1&-29.1672&802\\
\cline{2-10}
&3&1351976&-&-&665856&-&-&-&-\\
\end{tabular}}
\end{table}

\subsection{Minimizing a random complex quartic polynomial with unit-norm variables} 
The next example is to minimize a random complex quartic polynomial with unit-norm variables:
\begin{equation}\label{random:eq2}
\begin{cases}
\inf\limits_{\z\in\C^{s}} &[\z]_2^{\H}Q[\z]_2\\
\,\,\,\rm{s.t.}&|z_i|^2=1,\quad i=1,\ldots,s,
\end{cases}
\end{equation}
where $Q\in\H^{|[\z]_2|}$ is a random Hermitian matrix whose entries are selected with respect to the uniform probability distribution on $[0,1]$.

We approach \eqref{random:eq2} for $s=5,7,\ldots,15$ with the second and third HSOS relaxations. The related results are shown in Table \ref{tab2}. From the table, we see that the reformulation \eqref{rsos1} is about one magnitude faster than the reformulation \eqref{rsos0}, and again the speedup becomes more significant as the SDP size grows.

\begin{table}[htbp]
	\caption{Minimizing a random complex quartic polynomial with unit-norm variables.}\label{tab2}
	\renewcommand\arraystretch{1.2}
	\centering
 \resizebox{\linewidth}{!}{
\begin{tabular}{c|c|c|c|c|c|c|c|c|c}
\multirow{2}*{$s$}&\multirow{2}*{$d$}&\multicolumn{3}{c|}{\eqref{rsos0}}&\multicolumn{3}{c|}{\eqref{rsos1}}&\multicolumn{2}{c}{\eqref{hsos1}}\\
\cline{3-10}
&&$m_{\rm{sdp}}$&opt&time&$m_{\rm{sdp}}$&opt&time&opt&time\\
\hline
\multirow{2}{*}{5}&2&734&-24.4919&0.10&271&-24.4919&0.03&-24.4919&0.21\\
\cline{2-10}
&3&4474&-24.4919&2.34&1281&-24.4919&0.26&-24.4919&10.9\\
\hline
\multirow{2}{*}{7}&2&2202&-56.5289&0.65&869&-56.5289&0.16&-56.5289&1.15\\
\cline{2-10}
&3&21158&-46.7128&132&6637&-46.7128&7.44&-46.7128&520\\
\hline
\multirow{2}{*}{9}&2&5242&-114.342&4.62&2161&-114.342&0.73&-114.342&5.29\\
\cline{2-10}
&3&73312&-&-&24691&-81.2676&184&-&-\\
\hline
\multirow{2}{*}{11}&2&10718&-202.436&32.1&4555&-202.436&3.86&-202.436&30.0\\
\cline{2-10}
&3&206188&-&-&73327&-&-&-&-\\
\hline
\multirow{2}{*}{13}&2&19686&-338.041&126&8555&-338.041&12.7&-338.041&162\\
\cline{2-10}
&3&499438&-&-&185277&-&-&-&-\\
\hline
\multirow{2}{*}{15}&2&33394&-514.226&678&14761&-514.226&55.1&-514.226&705\\
\cline{2-10}
&3&1081514&-&-&414841&-&-&-&-\\
\end{tabular}}
\end{table}

\subsection{Minimizing a randomly generated sparse complex quartic polynomial over multi-spheres}\label{subsec3}
Given $l\in\N\backslash\{0\}$, we randomly generate a sparse complex quartic polynomial as follows: Let $f=\sum_{i=1}^lf_i\in\C[z_{1},\ldots,z_{5(l+1)},\overline{z}_{1},\ldots,\overline{z}_{5(l+1)}]$,\footnote{$\C[\z,\overline{\z}]$ denotes the ring of complex polynomials in variables $\z,\overline{\z}$.} where for all $i\in[l]$, $f_i=\overline{f}_i\in\C[z_{5(i-1)+1},\ldots,z_{5(i-1)+10},\overline{z}_{5(i-1)+1},\ldots,\overline{z}_{5(i-1)+10}]$ is a sparse complex quartic polynomial whose coefficients (real/imaginary parts) are selected with respect to the uniform probability distribution on $[-1,1]$. Then we consider the following CPOP:
\begin{equation}\label{scpop}
\begin{cases}
\inf\limits_{\z\in\C^{5(l+1)}}&f(\z,\overline{\z})\\
\,\quad\rm{s.t.}&\sum_{j=1}^{10}|z_{5(i-1)+j}|^2=1,\quad i=1,\ldots,l.
\end{cases}
\end{equation}
The sparsity in \eqref{scpop} can be exploited to derive a sparsity-adapted complex moment-HSOS hierarchy \cite{wang2022exploiting}.
We solve the second sparse HSOS relaxation of \eqref{scpop} for $l=40,80,\ldots,400$. The results are displayed in Table \ref{random2}. From the table we see that the reformulation \eqref{rsos1} is $1.5\sim2$ times as fast as the reformulation \eqref{rsos0}. Moreover, for this problem, solving the original complex SDP with {\tt Hypatia} is extremely slow probably due to the fact that the SDPs contain many PSD blocks.

\begin{table}[htbp]
\caption{Minimizing a randomly generated sparse complex quartic polynomial over multi-spheres.}\label{random2}
\renewcommand\arraystretch{1.2}
\centering
\resizebox{\linewidth}{!}{
\begin{tabular}{c|c|c|c|c|c|c|c|c}
\multirow{2}*{$l$}&\multicolumn{3}{c|}{\eqref{rsos0}}&\multicolumn{3}{c|}{\eqref{rsos1}}&\multicolumn{2}{c}{\eqref{hsos1}}\\
\cline{2-9}
&$m_{\rm{sdp}}$&opt&time&$m_{\rm{sdp}}$&opt&time&opt&time\\
\hline
40&23090&-98.9240&3.12&12529&-98.9240&2.06&-98.9240&886\\
\hline
80&46768&-197.577&12.6&25549&-197.577&8.07&-197.577&5433\\
\hline
120&70958&-292.024&30.1&38871&-292.024&19.0&$*$&$*$\\
\hline
160&94278&-389.652&45.9&51563&-389.652&30.7&$*$&$*$\\
\hline
200&117526&-482.684&84.5&64185&-482.684&37.7&$*$&$*$\\
\hline
240&140298&-578.896&130&76389&-578.896&59.5&$*$&$*$\\
\hline
280&162504&-671.047&173&89241&-671.047&65.4&$*$&$*$\\
\hline
320&187528&-766.403&206&102171&-766.403&88.5&$*$&$*$\\
\hline
360&210370&-866.771&291&114589&-866.771&147&$*$&$*$\\
\hline
400&233396&-963.137&297&127173&-963.137&138&$*$&$*$\\
\end{tabular}}
\end{table}

\subsection{Application to the AC-OPF problem}
The AC-OPF is a central problem in power systems, which aims to minimize the generation cost of an alternating current transmission network under physical and operational constraints. Mathematically, it can be formulated as the following CPOP:
\begin{equation}\label{opf}
\begin{cases}
\inf\limits_{V_i,S_k^g}&\sum_{k\in G}\left(\mathbf{c}_{2k}(\RR(S_{k}^g))^2+\mathbf{c}_{1k}\RR(S_{k}^g)+\mathbf{c}_{0k}\right)\\
\,\,\,\textrm{s.t.}&\angle V_r=0,\\
&\mathbf{S}_{k}^{gl}\le S_{k}^{g}\le \mathbf{S}_{k}^{gu},\quad\forall k\in G,\\
&\boldsymbol{\upsilon}_{i}^l\le|V_i|\le \boldsymbol{\upsilon}_{i}^u,\quad\forall i\in N,\\
&\sum_{k\in G_i}S_k^g-\mathbf{S}_i^d-\mathbf{Y}_i^{sh}|V_{i}|^2=\sum_{(i,j)\in E_i\cup E_i^R}S_{ij},\quad\forall i\in N,\\
&S_{ij}=(\overline{\mathbf{Y}}_{ij}-\mathbf{i}\frac{\mathbf{b}_{ij}^c}{2})\frac{|V_i|^2}{|\mathbf{T}_{ij}|^2}-\overline{\mathbf{Y}}_{ij}\frac{V_i\overline{V}_j}{\mathbf{T}_{ij}},\quad\forall (i,j)\in E,\\
&S_{ji}=(\overline{\mathbf{Y}}_{ij}-\mathbf{i}\frac{\mathbf{b}_{ij}^c}{2})|V_j|^2-\overline{\mathbf{Y}}_{ij}\frac{\overline{V}_iV_j}{\overline{\mathbf{T}}_{ij}},\quad\forall (i,j)\in E,\\
&|S_{ij}|\le \mathbf{s}_{ij}^u,\quad\forall (i,j)\in E\cup E^R,\\
&\boldsymbol{\theta}_{ij}^{\Delta l}\le \angle (V_i \overline{V}_j)\le \boldsymbol{\theta}_{ij}^{\Delta u},\quad\forall (i,j)\in E,\\
\end{cases}
\end{equation}
where $V_i$ is the voltage, $S_k^{g}$ is the power generation, $S_{ij}$ is the power flow (all are complex variables; $\angle\cdot$ stands for the angle of a complex number) and all symbols in boldface are constants. Notice that $G$ is the collection of generators and $N$ is the collection of buses. For a full description on the AC-OPF problem, we refer the reader to \cite{baba2019} as well as \cite{bienstock2020}. 

We select test cases from the AC-OPF library \href{https://github.com/power-grid-lib/pglib-opf}{PGLiB-OPF} \cite{baba2019}.
For each case, we solve the minimal relaxation step of the sparse HSOS hierarchy \cite{wang2022exploiting}. The results are displayed in Table \ref{ac-opf1}. From the table, we again see that the reformulation \eqref{rsos1} is several ($1.4\sim5$) times as fast as the reformulation \eqref{rsos0}.

\begin{table}[htbp]
\caption{The results for the AC-OPF problem. $s$: the number of CPOP variables; $t$: the number of CPOP constraints.}\label{ac-opf1}
\renewcommand\arraystretch{1.2}
	\centering
        \resizebox{\linewidth}{!}{
\begin{tabular}{c|c|c|c|c|c|c|c|c}
\multirow{2}*{Case}&\multirow{2}*{$s$}&\multirow{2}*{$t$}&\multicolumn{3}{c|}{\eqref{rsos0}}&\multicolumn{3}{c}{\eqref{rsos1}}\\
\cline{4-9}
&&&$m_{\rm{sdp}}$&opt&time&$m_{\rm{sdp}}$&opt&time\\
\hline
14\_ieee&19&147&2346&$1.9940\text{e}3$&0.19&422&$1.9940\text{e}3$&0.10\\
\hline
30\_ieee&36&297&4828&$8.1959\text{e}3$&0.73&836&$8.1960\text{e}3$&0.37\\
\hline
30\_as&36&297&4828&$5.0371\text{e}2$&0.55&836&$5.0371\text{e}2$&0.24\\
\hline
39\_epri&49&361&5270&$1.3568\text{e}5$&0.74&966&$1.3579\text{e}5$&0.54\\
\hline
89\_pegase&101&1221&57888&$9.4098\text{e}4$&63.6&10262&$9.4101\text{e}4$&15.1\\
\hline
57\_ieee&64&563&11102&$3.6644\text{e}4$&2.36&2008&$3.6644\text{e}4$&1.06\\
\hline
118\_ieee&172&1325&25374&$9.3216\text{e}4$&8.27&4471&$9.3216\text{e}4$&2.68\\
\hline
162\_ieee\_dtc&174&1809&64874&$1.0492\text{e}5$&43.4&11327&$1.0495\text{e}5$&13.8\\
\hline
179\_goc&208&1827&25712&$6.0859\text{e}5$&10.3&4368&$6.0860\text{e}5$&3.57\\
\hline
240\_pserc&383&3039&52172&$2.8153\text{e}6$&31.9&9243&$2.8170\text{e}6$&10.7\\
\hline
300\_ieee&369&2983&53946&$5.3037\text{e}5$&40.6&9647&$5.3037\text{e}5$&10.6\\
\hline
500\_goc&671&5255&90502&$3.9697\text{e}5$&89.8&15918&$3.9697\text{e}5$&25.4\\
\hline
588\_sdet&683&5287&79362&$1.9799\text{e}5$&91.7&13933&$1.9749\text{e}5$&21.3\\
\hline
793\_goc&890&7019&104978&$1.1194\text{e}5$&105&18536&$1.1222\text{e}5$&31.5\\
\hline
1888\_rte&2178&18257&280580&$1.2537\text{e}6$&939&47205&$1.2545\text{e}6$&180\\
\hline
2000\_goc&2238&23009&455530&$9.1876\text{e}5$&2087&77974&$9.1881\text{e}5$&439\\
\end{tabular}
}
\end{table}

\vspace{1em}
{\bf Summary of the benchmark.} The benchmark confirms our theoretical findings and demonstrates that our new real reformulation of complex SDPs is indeed more efficient than the classical real reformulation. Especially, the speedup factor gets even bigger as the SDP size grows and may reach several magnitudes on certain instances. It should be also noted that solving the original complex SDP with {\tt Hypatia} is even slower than solving the classical real reformulation with {\tt MOSEK}, probably because {\tt MOSEK} is commercial and the implementations of SDP algorithms based on real numbers are more mature and robust.

\section*{Acknowledgments}
The authors would like to thank Jurij Vol\v{c}i\v{c} for helpful comments on an earlier preprint of this note.

\section*{Funding}
This work was jointly funded by National Key R\&D Program of China under grant No. 2023YFA1009401 \& 2022YFA1005102, the Strategic Priority Research Program of the Chinese Academy of Sciences XDB0640000 \& XDB0640200, and the National Natural Science Foundation of China under grant No. 12201618 \& 12171324.

\section*{Conflict of interest}
The authors declare that they have no conflict of interest.

\section*{Data availability}
The author confirms that all data generated or analysed during this study are included in this article.

\bibliographystyle{siamplain}
\bibliography{refer}
\end{document}

%% file: ex_shared.tex
\usepackage{lipsum}
\usepackage{amsfonts}
\usepackage{graphicx}
\usepackage{epstopdf}
\usepackage{algorithmic}
\ifpdf
  \DeclareGraphicsExtensions{.eps,.pdf,.png,.jpg}
\else
  \DeclareGraphicsExtensions{.eps}
\fi

\newsiamremark{remark}{Remark}
\newsiamremark{hypothesis}{Hypothesis}
\crefname{hypothesis}{Hypothesis}{Hypotheses}
\newsiamthm{claim}{Claim}

\headers{A more efficient reformulation of complex SDP as real SDP}{Jie Wang}

\title{A more efficient reformulation of complex SDP as real SDP\thanks{Submitted to the editors DATE.
}}

\author{Jie Wang\thanks{Academy of Mathematics and Systems Science, Chinese Academy of Sciences, Beijing, China
  (\email{wangjie212@amss.ac.cn}, \url{https://wangjie212.github.io/jiewang/})}}

\usepackage{amsopn}
